\documentclass{amsart}
\usepackage{amsmath}  
\usepackage{amssymb}
\usepackage{amsthm}  
\usepackage{amscd} 
\usepackage{amsfonts} 
\usepackage{enumerate}
\usepackage{color}

\usepackage{inputenc}
\theoremstyle{plain}
\newtheorem{thm}{Theorem}[section]

\newtheorem{lemma}[thm]{Lemma}
\newtheorem{cor}[thm]{Corollary}

\theoremstyle{definition}

\theoremstyle{remark}
\newtheorem{rem}{Remark}

\numberwithin{equation}{section}

\newcommand{\R}{\mathbf{R}}

\newcommand{\G}{\mathrm G}

\title[Holomorphic isometric embeddings]
{Holomorphic isometric embeddings of  complex Grassmannians into quadrics: The general case}

\author{Oscar MACIA, Yasuyuki NAGATOMO}

\address{Department of Mathematics, UNIVERSITY OF VALENCIA, 
C. Dr. Moliner, S/N, (46100) Burjassot, SPAIN} 
\email{oscar.macia@uv.es}

\address{Department of Mathematics, MEIJI UNIVERSITY, 
Higashi-Mita, Tama-ku, Kawasaki-shi, Kanagawa 214-8571, JAPAN} 
\email{yasunaga@meiji.ac.jp}



\begin{document}

\begin{abstract}
The present article studies holomorphic isometric embeddings
of  arbitrary complex Grassmannians into quadrics, generalising
results in \cite{MaNa21}. 
The moduli spaces of these embeddings up to 
gauge and image equivalence are discussed using a generalisation of do Carmo--Wallach theory.
\end{abstract}
\maketitle


\section{Introduction}

The topic of holomorphic isometric embeddings between complex manifolds has a long history in complex geometry. For the enriched subclass of K\"ahler manifolds, the earliest work dates back to Calabi's seminal 1953 paper \cite{Cal}, and the field has remained  active  since its inception.

More recently, a  general theory describing the moduli space of holomorphic isometric embeddings of K\"ahler manifolds into Grassmannians  has been developed  in \cite{Na13, Na15}. The foundations of this theory, a generalisation of ideas of Takahashi \cite{TTaka}, and of do Carmo \and Wallach \cite{DoC-Wal},   are  deeply rooted in the principles of gauge theory. Therefore, techniques  based on vector bundles and the representation theory of Lie groups play a prominent 
{\it r\^ole} in it. 

As  a practical application of this generalisation of do Carmo--Wallach theory, the description of the  moduli  space when the embedded manifold is
a complex projective space and the target is a complex hyperquadric (identified as an appropriate real Grassmannian) is by now well--understood: The $\mathbf C \mathrm P^1$ case has been discussed at length in \cite{MNT, MacNag}, and the generalisation to $\mathbf C \mathrm P^n$ has been recently  considered in \cite{Na21}.

In spite of this, the moduli spaces of holomorphic isometric embeddings of complex Grassmannian manifolds into quadrics are much less familiar. To improve this situation, the authors studied first a simplified version of the problem in  \cite{MaNa21}:  the moduli of holomorphic isomorphic embeddings of the complex Grassmann manifolds 
$\mathrm{Gr}_m(\mathbf{C}^{m+2})$ of all two--codimensional linear subspaces in $\mathbf{C}^{m+2}$ into quadrics. This particular Hermitian
 symmetric space is in itself worthy of attention from the intrinsically geometric point of view, as the
unique compact, K\"ahler, quaternionic K\"ahler manifold with positive scalar curvature \cite{Besse, Salamon, Wolf}. 
On the other hand that special case is simple enough to provide a model to the intricate representation--theoretic computations that underpin the general theory.  In particular, it provides a pathway for the decomposition of  the space ${\rm GS}(\mathfrak m V_0,V_0)$  (\cite{MaNa21}, \S 2.5) into irreducible representations. 
 
 While Young {\it tableaux} and Littlewood--Richardson rules are recognised as the right tools  to systematically determine the decomposition of tensor products of  ${\rm SU}(n)$ representations \cite{Ful}, these are nonetheless insufficient to irreducibly decompose ${\rm GS}(\mathfrak m V_0,V_0),$ since Littlewood did not laid  a straightforward method to determine which terms of the decomposition belong to the symmetric or skewsymmetric parts \cite{Littlewood}.  
Though more modern algebro--combinatorial techniques (aka {\it domino  tableaux}, \cite{domino}) have been
developed which would allow to decompose into
irreducible representations while keeping track of the general symmetry or skewsymmetry, we have kept with our original approach based on  working out the action of the centre of the isotropy subgroup on the highest and lowest--weight vectors, in a way
 similar to, if more complicated than, that of \cite{MaNa21}
 (compare \S 3 with {\it loc.cit.} Lemmas 4.1 to 4.4). The special conditions which, due to the low dimensionality, allow the representations to be decomposed with minimum difficulty disappear in this article, and the parameterisation of the highest--weight modules becomes a delicate issue which is solved in detail in Lemma 3.1.

 This article, being a sequel to \cite{MNT,MacNag,Na21} and a generalisation of the discussion in \cite{MaNa21},  is based on the same general theory as its predecessors. Therefore, the authors have not included material that could easily be found elsewhere.  
This theoretical minimum (natural evaluations and identifications, image and gauge equivalences, vector bundles on Hermitian symmetric spaces, standard maps, and  do Carmo-Wallach theory, together with some motivation) can be found described in detail, for example, in \cite{MaNa21}, \S\S 2 and 3, and the reader would do well to consult it to settle the definitions and context of the present work.

The structure of the article is the following:
After some preliminary remarks which have been briefly collected in section 2, and which serve as a bridge between the general theory mentioned above and the concrete results of the present work, section 3 is devoted to the detailed computation of the irreducible $\mathrm{SU}(p+q)$ representations which make up  the moduli,
 culminating in Theorem \ref{gmod1} where its geometry is determined. The relation between the moduli up to gauge and up to image equivalence is discussed in Corollary \ref{imod}.
 Lastly, Section 4 deals with the special case  $p=q,$ that is, embeddings of $\mathrm{Gr}_q(\mathbf C^{2q}),$ which generalises the example of the compactified, complexified Minkowski space $\mathrm{Gr}_2(\mathbf C^4)$ relevant in
 certain areas of high energy theoretical physics.
 
 The authors intend, in a subsequent report, to further generalise the contents of this article  to study the case of holomorphic isometric embeddings of  flag varieties into complex hyperquadrics.

\indent\paragraph{\bf Acknowledgements}
The work of O.M. is partially supported by  the AEI (Spain) and FEDER project  PID2019-105019GB-C21, and by the  GVA project AICO 2021 21/378.01/1.
The work of Y.N. is supported by JSPS KAKENHI Grant Number 21K03236.
The authors would like to thank  Profs. Simon Salamon (K.C.London), Andrew Swann (Aarhus U.), Uwe Semmelmann (Stuttgart U.) and Rowena Paget (Kent U.), for useful conversations.

\section{Preliminary remarks}
From this point, familiarity with the generalisation of the do Carmo--Wallach theory, at the level of \cite{MaNa21}, \S\S 2 and 3,  is assummed throughout. The general picture of the theory is the following:
Suppose that a vector bundle  $V\to M$ and a finite--dimensional subspace $W$ of the space of sections $\Gamma(V)$ is given. If the evaluation homomorphism $ev : M\times W \to V$  is surjective (so that $V\to M$ is  globally generated by $W$), assigning to each $x\in M$ the $p$--plane $\text{Ker}\,ev_x\subset \mathrm{Gr}_p(W)$
defines an induced map $f:M\to \mathrm{Gr}_p(W).$ Using it to pull 
the natural exact sequence over $\mathrm{Gr}_p(W)$ back to $M$  it can be proved that $V\to M$ is naturally identified with $f^*Q\to M$ where $Q\to \mathrm{Gr}_p(W)$ stands for the universal quotient bundle. In case that one deals with vector bundles equipped with metrics and connections, additional differential conditions need to be considered (e.g., the gauge condition which states that the connection on $V\to M$ must be gauge equivalent to the connection on $f^*Q\to M$). In the present article $M$ stands for a compact,
irreducible, Hermitian symmetric space, and $V\to M$ will be substituted by a complex  homogeneous line bundle. The subspace of sections $W\subset\Gamma(V)$ will be the space of holomorphic sections $H^0(V)\subset \Gamma(V)$ since in this context this is the space which determines the {\it standard maps} (the general definition for standard map being those induced by sets of sections with the same eigenvalue for the Laplace operator acting on sections).
Regard $W$ as a real vector space  and fix a $\mathbf R^{n+2}$ subspace of $W,$ with inclusion map $\iota.$ This being a real space of sections $\mathbf{R}^{n+2}$ also generates $V\to M$ globally. The generalisation of the
theory of do Carmo--Wallach shows that this $\mathbf R^{n+2}$ can be regarded as the orthogonal complement of the kernel of a certain
positive semi-definite symmetric endomorphism $T$ of $W,$ which is indeed positive definite on $\mathbf R^{n+2}.$  Specifically, 

\begin{thm} \label{GenDW}

Let $(L,h)$ be a fixed Hermitian line bundle on $\mathrm{G}/\mathrm{K},$ and let
$f:{\rm G/K} \to {\rm Gr}_n(\mathbf R^{n+2})$ be a full holomorphic map
satisfying the {\it gauge condition for} $(L,h).$ 
Let $W$ denote $H^0({\rm G/K},L)$ regarded as a real vector space equipped with a complex structure.
Then, there is a 
positive semi-definite symmetric endomorphism 
$T\in \text{\rm S}\,(W)$ 
such that the pair $(W,T)$ satisfies the following three conditions:\\
\begin{enumerate}[(a)]
\item The vector space $\mathbf{R}^{n+2}$ is a subspace of $W$ with the inclusion 
$\iota:\mathbf{R}^{n+2} \to W$ preserving the inner products, and 
$L\to M$ is globally generated by $\mathbf{R}^{n+2}$.
\item As a subspace, $\mathbf{R}^{n+2}=\text{\rm Ker}\,T^{\bot}$ and 
the restriction of $T$  is a positive symmetric endomorphism of $\mathbf{R}^{n+2}$. 
\item The endomorphism $T$ satisfies 
\begin{equation*}\label{HDW 2}
\left(T^2- I\!d_W, \G\mathrm{S}(V_0, V_0)\right)_{S}=0, 
\qquad
\left(T^2, \G\mathrm{S}(\mathfrak m V_0, V_0)\right)_{S}=0.
\end{equation*}
\end{enumerate}
\bigskip

If $\iota^{\ast}:W \to \mathbf{R}^{n+2}$ denotes the adjoint linear map of 
$\iota:\mathbf{R}^{n+2} \to W$, 
then $f:{\rm G/K} \to \mathrm{Gr}_n(\mathbf{R}^{n+2})$ is realized as 
 \begin{equation}\label{HDW5} 
f\left([g]\right)=\left(\iota^{\ast}T\iota \right)^{-1}\left(f_0\left([g]\right)\cap \text{\rm Ker}\,T^{\bot}\right), 
\end{equation}
where the orientation of 
$\left(\iota^{\ast}T\iota \right)^{-1}\left(f_0\left([g]\right)\cap \text{\rm Ker}\,T^{\bot}\right)_{[g]}$ 
is given by the orientation of $L_{[g]}$ and $\mathbf{R}^{n+2}$. 
Moreover, if the orientation of $\text{Ker}\,T$ is fixed, then 
we have a unique holomorphic totally geodesic embedding of 
$\mathrm{Gr}_n(\mathbf{R}^{n+2})$ into $\mathrm{Gr}_{n^{\prime}}(W)$ 
by $\iota\left(\mathbf{R}^{n+2}\right)=\text{\rm Ker}\,T^{\bot}$, 
where $n^{\prime}=n+\text{\rm dim}\,\text{\rm Ker}\,T$
and a bundle isomorphism   
$\left(ev \circ \iota\circ\left(\iota^{\ast} T \iota\right)\right)^{\ast}
:L\to f^{\ast}Q$ as the natural identification by $f$.  
Such two maps $f_i$, $(i=1,2)$ are gauge equivalent if and only if 
$
\iota^{\ast}T_1\iota=\iota^{\ast}T_2\iota, 
$
where $T_i$ and $\iota$ correspond to $f_i$ in \eqref{HDW5}, respectively.\\ 

Conversely, 
suppose that a vector space $\mathbf{R}^{n+2}$ with an inner product and an orientation, and 
a positive semi-definite symmetric endomorphism 
$T\in \text{\rm End}\,(W)$ 
satisfying 
conditions (a),(b),(c) are given.  
Then we have a unique holomorphic totally geodesic embedding of 
$\mathrm{Gr}_n(\mathbf{R}^{n+2})$ into $\mathrm{Gr}_{n^{\prime}}(W)$ after fixing the orientation of 
$\text{Ker}\,T$  
and 
the map $f:G/K \to \mathrm{Gr}_{p}(\mathbf{R}^{n+2})$ defined by \eqref{HDW5} 
is a full holomorphic map into $\mathrm{Gr}_n(\mathbf{R}^{n+2})$ 
satisfying the gauge condition with bundle isomorphism $L\cong f^{\ast}Q$ 
as the natural identification.

\end{thm}

Theorem \ref{GenDW}, originally proved in \cite{Na13}, gives the representation theoretic conditions that
such endomorphism $T$ must satisfy (condition (c) in the theorem). Each such $T$ determines a new holomorphic embedding $\mathrm{Gr}_n(\mathbf R^{n+2})\to \mathrm{Gr}_{n'}(W)$ where $n'=n+\dim\ker T.$ The theoretical  expression for the holomorphic isometric embedding $f$ constructed  from the standard map $f_0,$ the symmetric endomorphism $T,$ and the inclusion map $\iota$ is given by Equation (\ref{HDW5}). Each symmetric endomorphism $T$ determines a unique gauge equivalence class of maps.
 Therefore, knowledge of the representation spaces where $T$ lives  determines all the relevant differential geometric information of the moduli space up to gague equivalence $\mathcal M_k,$ in particular its dimension, and the meaning of the boundary points in the compactification.

In the present context, Theorem \ref{GenDW} allows 
to determine the  moduli space $\mathcal M_k$ of holomorphic isometric embeddings of degree $k$ modulo  the gauge equivalence of maps due to the special characterisation,  in terms of vector bundles and pull--back connections, to which  these maps are subject in the generalisation of do Carmo--Wallach theory:
Let $\mathcal Q_n$ denote the complex quadric hypersurface in $\mathbf C P^{n+1},$ identified with $\mathrm{Gr}_n(\mathbf{R}^{n+2})$ as in \cite{Kob}, pp.278--282, and let $Q\to\mathcal{Q}_n$ be the universal quotient bundle over $\mathcal Q_n.$ Analogously,
let $\widetilde Q\to\mathrm{Gr}_p(\mathbf{C}^{p+q})$ denote the universal quotient bundle over $\mathrm{Gr}_p(\mathbf{C}^{p+q}),$
which satisfies $\mathcal O(1)\cong\wedge^q\widetilde Q.$ Now, it is a  fact in the general theory that $f:\mathrm{Gr}_p(\mathbf{C}^{p+q})\to\mathcal Q_n$ is a holomorphic isometric embedding
of degree $k$ if and only if the pull--back by $f$ of the canonical connection on $Q\to \mathcal Q_n$ coincides with the canonical connection on $\mathcal O(k)\to\mathrm{Gr}_p(\mathbf{C}^{p+q}),$
 that is, if and only if $f$ satisfies the gauge condition for $\left( \mathcal O(k), h_k \right),$ where $h_k$ denotes the standard Einstein--Hermitian metric on $\mathcal O(k)\to \mathrm{Gr}_p(\mathbf{C}^{p+q}).$
Since the canonical connection on $\left( \mathcal O(k), h_k \right)$ is  the Hermitian connection, 
Theorem \ref{GenDW} allows 
to determine the  moduli space $\mathcal M_k.$ 

In general terms, if $(L,h)$ is a given Hermitian holomorphic line bundle over a K\"ahler manifold,  $H^0(M,L)$ inherits a complex structure commuting
with the evaluation  $ev:\underline{H^0(M,L)}\to L.$
Any complex subspace of $H^0(M,L)$ can also be regarded as a real
vector space $U$ equipped with this complex structure and an inner product inherited from the $L^2$--Hermitian product on $H^0(M,L),$ and the couple $(L\to M, U)$ induces a map of
$M$ into a quadric. If this map is a holomorphic isometric immersion, commutativity between the complex structure and
the evaluation map allows to regard it as a holomorphic isometric immersion into a totally geodesic complex projective space contained in the target quadric. 
If the Kodaira embedding into a complex projective space, which is the induced map by $(\mathcal O(k)\to \mathrm{Gr}_p(\mathbf{C}^{p+q}),$ $\; H^0(\mathrm{Gr}_p(\mathbf{C}^{p+q}),\mathcal O(k)))$  is a holomorphic isometric embedding, then it is rigid  \cite{Cal}. When both the bundle and the space of sections  inducing the  
map 
are regarded as  
real and oriented, the Kodaira embedding can
also be  viewed  as a holomorphic isometric embedding into a quadric. 
Therefore, the characterisation of the moduli space of full holomorphic isometric embeddings $\mathrm{Gr}_p(\mathbf{C}^{p+q})\to \mathcal Q_n$ now follows from Theorem 6.10 of \cite{Na15},  

\begin{thm}
\label{quadmodcpx}
Let $(L,h)$ be a Hermitian holomorphic line bundle over a K\"ahler manifold 
$M$. 
Let 
$\mathcal M$ be the moduli space of 
full holomorphic isometric immersions of 
$M$ into a quadric $\mathrm{Gr}_n(\mathbf{R}^{n+2})$ 
with the gauge condition for $(L,h)$ modulo gauge equivalence relation of maps.   

Suppose that  $\mathbf{R}^{n+2}$ is a {\it complex}\, subspace of $H^0(M,L)$ and 
the inner product on $\mathbf{R}^{n+2}$ is compatible with the complex structure. 
If there exists $f \in \mathcal M$ such that the evaluation map 
$ev:\underline{\mathbf{R}^{n+2}}\to L$ by $f$ satisfies $J_L\circ ev=ev\circ J$,  
then $\mathcal M$ has the induced complex structure and 
is an open submanifold of a complex  
subspace of $\mathrm  H(\mathbf{R}^{n+2})^\perp$. 
\end{thm}

As a consequence of the previous discussion, the moduli up to gauge equivalence of maps, of full holomorphic isometric embeddings 
$\mathrm{Gr}_p(\mathbf{C}^{p+q})\to \mathcal Q_n$  satisfying the
  gauge condition for 
$(\mathcal O(k) \to \mathbf{C}P^m,h_k)$ 
 is an open submanifold which according to the general theory lies in the  orthogonal complement $\mathrm  H(\mathbf{R}^{n+2})^\perp$ of the complex  
subspace of Hermitian endomorphisms  of the space of symmetric endomorphisms $\mathrm{Sym}(\mathbf R^{n+2}),$ 
where $U\equiv \mathbf{R}^{n+2}=H^0\left(\mathrm{Gr}_p(\mathbf{C}^{p+q}), \mathcal O(k)\right)$. 

\section{Moduli spaces}

In the following paragraphs   
we compute   the specific $\mathrm{SU}(p+q)$ representations which, according to Theorem \ref{GenDW} and \S 2 make up the moduli space up to gauge equivalence. In order to apply the   
generalisation of the do Carmo--Wallach theory  a detailed understanding of the  symmetric endomorphisms of the space $W=H^0({\rm G/K},L)$  of holomorphic sections of certain line bundles over ${\rm G/K}$ is essential.

The Grassmannian of $p$--planes in $\mathbf C^{p+q}$ is described as a Hermitian symmetric space by the symmetric pair 
$(\mathrm G,\mathrm K)=\left({\rm SU}(p+q),\; 
\mathrm{S} \left(\mathrm{U} (p) \times \mathrm{U} (q)\right) \right)$ such that
\[\mathrm{Gr}_p(\mathbf C^{p+q})=\frac{{\rm SU}(p+q)}{\mathrm S (\mathrm U(p)\times \mathrm U(q))}\]
The isotropy group $\mathrm K={\rm S}(\mathrm U(p)\times \mathrm U(q))$  is given by matrices  of the form
\[\begin{pmatrix}
A & O \\
O & B
\end{pmatrix}, \quad A \in \text{U}(p), \,\,B \in \text{U}(q),\qquad |A||B|=1.
\]
Let us fix a maximal torus $T^n$ inside $\mathrm U(n)$ defined as the subgroup of diagonal matrices, 
\[ {\rm diag}(x_1,\dots,x_n)=\left(\begin{array}{ccc}
x_1 & & \\
       &\ddots & \\
       &        & x_n
       \end{array}
\right)\]
and denote by $\mathbf{V}_n(\lambda)$  
an irreducible complex representation of 
$\text{U}(n)$ with the highest weight 
$\lambda=(\lambda_1,\lambda_2, \cdots,\lambda_n)$, 
where $\lambda_1 \geq \lambda_2 \geq \cdots \geq \lambda_n$ are integers.  The action of the maximal torus on the highest--weight vector $\hat w\in \mathbf{V}_n(\lambda)$  is by definition
\[{\rm diag}(x_1,\dots,x_n)\cdot \hat w = x_1^{\lambda_1}\dots x_n^{\lambda_n}\hat w.\]
If $\mathbf{V}_n(\lambda)$ is regarded as an ${\rm SU}(n)$ representation,
it will  satisfy $\mathbf{V}_n(\lambda)=\mathbf{V}_n(\lambda+t(1,\dots,1)),$  $t\in\mathbf Z,$ so by convention we will
choose $t=-\lambda_n$ so the last weight becomes zero.
The {\it 
dominant integral weights} of ${\rm SU}(n),$ 
denoted by $\varpi_i\; (1\leq i \leq n-1),$ are defined by
$\lambda_1=\dots=\lambda_i=1,\; \lambda_{i+1}=\dots=\lambda_n=0.$ The ${\rm SU}(n)$ irreducible representation with highest weight $\varpi=\sum_i k_i\varpi_i$ will be denoted by $\mathbf{F}_n(\varpi).$ 
 The relation between  
 $\mathbf{F}_n(\varpi)$   and $\mathbf{V}_n(\lambda)$ is
\begin{equation}\label{FtoV}
\mathbf{F}_n\left(\sum_{i=1}^{n-1} k_i\varpi_i\right)
=\mathbf{V}_n\left(\sum_{i=1}^{n-1} k_i, \sum_{i=2}^{n-1} k_i, \cdots, k_{n-2}+k_{n-1},k_{n-1},0\right). 
\end{equation}
As it is customary,  the subindex $n$ in $F_n,\;V_n$ will be usually neglected when it is clear enough. If $\varpi=\sum_i k_i\varpi_i$ is a dominant integral weight, associated to the ${\rm SU}(n)$ representation $\mathbf{F}(\varpi),$ the action of an element of the maximal torus of $\mathrm U(n)$ on the highest  weight vector $\hat w\in \mathbf{F}(\varpi)$ is given by 
\[{\rm diag}(x_1,\dots,x_{n})\cdot \hat w=x_1^{\lambda_1}\dots x_{n-1}^{\lambda_{n-1}}\hat w\] 

The action of the torus of $\mathrm U(n)$ on the representation
dual to $\mathbf{V}(\lambda)$ is determined as follows: Since the dual representation of $\mathbf{V}(\lambda)$ has lowest weight equal to $-\lambda,$ the torus of ${\rm U}(n)$ acts on the lowest weight vector $\check w\in \mathbf{F}(\varpi)$ as it does on the highest weight vector of $\mathbf{F}(-\varpi ')$ with
\[
\varpi'=-\sum_{i=1}^{n-1} k_i\varpi_{n-i}=-\sum_{j=1}^{n-1} k_{n-j}\varpi_j. 
\] 
Therefore
\begin{eqnarray*}
{\rm diag}(x_1,\dots ,x_{n})\cdot \check w 
&=&
x_1^{-\sum_{j=1}^{n-1} k_{n-j}} 
x_2^{-\sum_{j=2}^{n-1} k_{n-j}}\cdots 
x_{n-2}^{-k_{2}-k_{1}} 
x_{n-1}^{-k_{1}} \check w\\
&=&
 x_1^{-\sum_{j=1}^{n-1} k_{j}} 
x_2^{-\sum_{j=1}^{n-2} k_{j}} \cdots  
x_{n-2}^{-k_{2}-k_{1}} 
x_{n-1}^{-k_{1}}\check w.
\end{eqnarray*}

Recalling that for our purposes $n=p+q,$ denote by $Y$  an element of the center ${\rm U}(1)$ of ${\rm K}=\mathrm  S(\mathrm U(p)\times \mathrm U(q))={\rm SU}(p)\times {\rm SU}(q)\times {\rm U}(1),$ the second equality holding only at the Lie algebra level:
\[
Y=\text{diag}
(\;
\underbrace{y^{\frac{1}{p}},\dots,y^{\frac{1}{p}}}_{p}, 
\underbrace{y^{-\frac{1}{q}}, \dots, y^{-\frac{1}{q}}}_{q}\; ). 
\] 
Hence, the action of $Y$  on $\check w$ is given by 

\begin{eqnarray}\label{actionofY}
Y\cdot\check w & = & 
\underbrace{y^{-\frac{1}{p}\sum_{j=1}^{p+q-1} k_{j}}
y^{-\frac{1}{p}\sum_{j=1}^{p+q-2} k_{j}} 
\cdots 
y^{-\frac{1}{p}(k_{q+1}+k_q+\dots+k_1)}
y^{-\frac{1}{p}(k_q+\dots+k_1)}}_p   \\
 & &\times \;
\underbrace{y^{\frac{1}{q}\sum_{j=1}^{q-1} k_{j}}
y^{\frac{1}{q}\sum_{j=1}^{q-2}k_j}\dots
y^{\frac1q (k_2+k_1)}
y^{\frac{1}{q}k_1} }_{q-1}\check w \nonumber \\
&=&
 y^{-\frac{p}{p}(k_q+\dots+k_1)-\frac{(p-1)k_{q+1}}{p}+\frac{(p-2)k_{q+2}}{p}-\dots - \frac{2 k_{p+q-2}}{p}-\frac{k_{p+q-1}}{p}} \nonumber \\
&  & \times\; 
y^{\frac{q-1}{q}k_1 + \frac{(q-2)k_2}{q} +\frac{(q-3)k_3}{q}+\dots
+\frac{2k_{q-2}}{q}+\frac{k_{q-1}}{q}}\check w \nonumber \\
&=& y^{-\frac{1}{q}k_1-\frac{2}{q}k_2-\dots - \frac{q-1}{q}k_{q-1}-k_q-\frac{1}{p}\sum_{i=1}^{p-1}(p-i)k_{q+i}}\check w \nonumber \\
&=& y^{-\frac{1}{q}\sum_{i=1}^q i k_i -\frac{1}{p}\sum_{j=1}^{p-1}k_{q+j}}\check w. \nonumber
\end{eqnarray}

In the present situation (compare with the setting in \cite{MaNa21}) $H^0(M;\mathcal O(k))=\mathbf{F}(k \varpi_q)$   
in virtue of the Bott--Borel--Weil Theorem \cite{Bott}. 
The moduli space of interest is  now a subset of   
$S^2(\mathbf{F}(k\varpi_q)),$ which was denoted by $\mathrm  H(W)^\perp $ in \S 2.
Since $S^2(\mathbf{F}(k\varpi_q))\subset \otimes^2 \mathbf{F}(k\varpi_q)$
we first apply  Littlewood--Richardson rules on Young tableaux    with $q$ rows and $k$ columns \cite{Ful}, to decompose $\otimes^2 \mathbf{F}(k\varpi_q)$ into irreducible representations under the action of $\mathrm{SU}(p+q).$
  We have that

\begin{lemma}
\begin{eqnarray}\label{ClebschGordan}
&& \mathbf{F}(k\varpi_q)  \otimes_{\mathbf C} \mathbf{F}(k\varpi_q) \\
 &=& 
\bigoplus_{i_1,i_2,\dots, i_q=0}^{k\geq \sum_{\alpha=1}^q i_\alpha}\mathbf{V}(2k-i_q,\;2k-(i_q+i_{q-1}),\;\dots, 2k-\sum_{\alpha=1}^q i_\alpha, 
\nonumber \\
& & \qquad \qquad \quad \sum_{\alpha=1}^q i_\alpha, \sum_{\alpha=2}^q i_\alpha, \dots, i_{q-1}+i_q, i_q, 0,\dots,0)  \nonumber \\
&=&
\bigoplus_{i_\alpha=0}^{k\geq \sum_{\alpha=1}^q i_\alpha }
\mathbf{F}(\underbrace{i_{q-1},i_{q-2},\dots, i_2, i_1}_{q-1},\; 2k-2\sum_{\alpha=1}^q i_\alpha, \;\underbrace{i_1, i_2, \dots, i_{q-1}, i_q}_{q},0,\dots,0) \nonumber
\end{eqnarray}

\end{lemma}

\begin{rem}  Before entering into the details of the proof, let us introduce some
notation. 
Recall that by Eqn.(\ref{FtoV}) the weight $ k \varpi_q$ is 
$\lambda=(k,\dots,k,0,\dots,0)$ containing $q$ $k$'s, followed
by $p-1$ zeroes. To the representation $\mathbf{V}(\lambda)$ with this highest weight  it is associated a Young {\it tableau} \cite{Ful}, a diagram with $q$ rows, each with $k$ columns, or `boxes'. The tensor product $\mathbf{V}(\lambda)\otimes \mathbf{V}(\lambda)$ becomes $\oplus_{\mu}\mathbf{V}(\mu)$ where each of the admissible
$\mu$ is determined diagramatically
using the Littlewood--Richardson rules \cite{Littlewood}, which give a list of admissible rearrangements of the $2 q k$ boxes appearing in the Young {\it tableau} of the $\mathbf{V}(\lambda)$ factors of the tensor product. 
In order to keep track of  which boxes of the second diagram are added to each row of the first diagram we introduce variables $x^i_j,$ where  $1\leq i\leq q,$ and $1\leq j \leq 2q$ which stands for the number $x$ of boxes from the $i$-th row  of the second diagram which are to be attached to the right to  the $j$-th row of the first diagram: eg the quantity $k+x^1_5 + x^3_5$ represents the total number of boxes in row $5$ after $x^1_5$ boxes from the first row  of the second diagram and $x^3_5$ boxes from the third row  of the second diagram have been attached to the right of the $k$ boxes of the fifth row  in the first diagram.

With this notation, the relevant combinatorial Littlewood rules determining admissible arrangements can be stated as follows
\begin{eqnarray}
\label{lrr1} x^q_q\leq x^{q-1}_{q-1}\leq \dots\leq x^1_1 & \\
\label{lrr2} x^i_i+ x^i_{q+1}\leq x^{i-1}_{i-1}\qquad & (1\leq i\leq q)\\
\label{lrr3}x^i_i+\sum_{r=1}^s x^i_{q+r}\leq x^{i-1}_{i-1}+\sum_{r=1}^{s-1}x^i_{q+r} &  (2\leq s\leq i-1)\\
\label{lrr4}x^i_i+\sum_{r=1}^ix^i_{q+r}=x^{i-1}_{i-1}+\sum_{r=1}^{i-1}x^{i-1}_{q+r}=k & \\
\label{lrr5} \sum_{i=1}^q x^i_j \leq \sum_{i=1}^q x^i_{j-1} & (1\leq j \leq 2q)
\end{eqnarray} this last rule sometimes referred to as  the `row--length rule'.
\end{rem}

\begin{proof}
After the boxes of the right diagram have been added to the boxes of the left diagram, consider the number of boxes in the $q$-th row. 
The equation (\ref{lrr4}) follows from the total number of boxes in each row  before the rearrangement. Applying  (\ref{lrr4}) for $i=q$ (that is equal to consider the $q$-th and $q-1$-th rows) we obtain 
\begin{equation}\label{xqq} x^q_q=k-\sum_{r=1}^q x^q_{q+r},\qquad x^{q-1}_{q-1}=k-\sum_{r=1}^{q-1}x^{q-1}_{q+r}\end{equation}Substituting the values obtained for $x^q_q,\; x^{q-1}_{q-1}$ into (\ref{lrr3}) for $s=q-1,$ leads to
\[\left(k-\sum_{r=1}^q x^q_{q+r}\right)+\sum_{r=1}^{q-1}x^q_{q+r}\leq \left(k-\sum_{r=1}^{q-1}x^{q-1}_{q+r}\right)+\sum_{r=1}^{q-2}x^{q-1}_{q+r}\] after cancellation, this leads to
\[k-x^q_{2q}\leq k-x^{q-1}_{2q-1}\Longrightarrow x^q_{2q}\geq x^{q-1}_{2q-1}\]
By the row--length rule (\ref{lrr5}) it is implied that $x^q_{2q}\leq x^{q-1}_{2q-1}.$ Alltogether we are lead to
\begin{equation}x^{q-1}_{2q-1}=x^q_{2q}\end{equation}

Repeating this same argument now for the $(q-1)$-th row, that is fixing $i=q-1$ in (\ref{lrr4}), we have
\[x^{q-1}_{q-1}=k-\sum_{r=1}^{q-1}x^{q-1}_{q+r},\qquad x^{q-2}_{q-2}=k-\sum_{r=1}^{q-2}x^{q-2}_{q+r}\] and substituting back into
(\ref{lrr3}) for $s=q-2$ leads to
\[
\left(k-\sum_{r=1}^{q-1}x^{q-1}_{q+r}\right)+\sum_{r=1}^{q-2}x^{q-1}_{q+r}\leq \left(k-\sum_{r=1}^{q-2}x^{q-2}_{q+r}\right)+\sum_{r=1}^{q-3}x^{q-2}_{q+r}
\] which after cancellations simplifies to
\[k-x^{q-1}_{2q-1}\leq k-x^{q-2}_{2q-2}\Longrightarrow x^{q-1}_{2q-1}\geq x^{q-2}_{2q-2}\]
Since the row--length rule (\ref{lrr5}) implies that $x^{q-2}_{2q-2}\geq x^{q-1}_{2q-1}$ therefore
\begin{equation}
x^{q-2}_{2q-2}=x^{q-1}_{2q-1}.
\end{equation}
Steady iteration of the same argument leads to a series of relations similar to the previous ones
\begin{equation}
x^1_{q+1}=x^2_{q+2}=\dots=x^{q-2}_{2q-2}=x^{q-1}_{2q-1}=x^q_{2q}
\end{equation}
which allows to set $x^q_{2q}$ as a parameter for the arrangement of the other boxes.

Next, substituting the expressions (\ref{xqq}) into (\ref{lrr3}), this time fixing $s=q-2$ we obtain
\[\left(
k-\sum_{r=1}^q x^q_{q+r} \right)+\sum_{r=1}^{q-2}x^q_{q+r}\leq 
\left(
k-\sum_{r=1}^{q-1}x^{q-1}_{q+r}
\right)+\sum_{r=1}^{q-3}x^{q-1}_{q+r}\] that is
\[k-x^q_{2q-1}-x^q_{2q}\leq k-x^{q-1}_{2q}\leq k-x^{q-1}_{2q-1}-x^{q-1}_{2q-2}\] Using the previously deduced equality of $x^q_{2q}=x^{q-1}_{2q-1}$ and simplifying one gets $x^q_{2q-1}\geq x^{q-1}_{2q-2}.$ Then, using the row--length rule (\ref{lrr5}), together with the fact that $x^{q-2}_{2q-2}=x^{q-1}_{2q-1}$ we have that $x^{q-1}_{2q-2}\geq x^q_{2q-1},$ hence
\begin{equation}
x^{q-1}_{2q-2}=x^q_{2q-1}.
\end{equation}
Iteration of the argument as before leads to the identification of $x^q_{2q-1}$ as a new parameter through relations
\begin{equation}
x^2_{q+1}=x^3_{q+2}=\dots= x^{q-1}_{2q-2}=x^q_{2q-1}
\end{equation}
Orderly application of the same method each time for lower $s$ in (\ref{lrr3}) leads to a sequence of relatios in which $\{x^q_{q+r}\;:\;r=1,\dots,q\}$ become the parameters annd determine all other entries
\begin{equation}
x^\lambda_{q+\mu}=x^q_{q+r}\quad \text{if} \quad (q+\mu)-\lambda=r,\qquad 1\leq \lambda \leq \mu\leq q.
\end{equation}
The remaining variables $x^i_i,\; i=1,2,\dots,q$ are determined by the corresponding Littlewood rule (\ref{lrr4}), hence
\[x^1_1 = k-x^q_{2q},\qquad x^2_2= k-x^q_{2q-1}-x^q_{2q},\qquad \dots\]
etc., that is
\begin{equation}
x^i_i=k-\sum_{r=q+1-i}^q x^q_{q+r}
\end{equation}
The corresponding irreducible decomposition for $\mathbf{F}(k\varpi_q)\otimes_{\mathbf C} \mathbf{F}(k\varpi_q),$ that is, $\mathbf{V}(\lambda)\otimes_{\mathbf C} \mathbf{V}(\lambda)$ where
$\lambda=(k,\dots,k,0,\dots,0)$ will be 
$\oplus_{\lambda'} \mathbf{V}(\lambda')$ where $\lambda'$ is any of the admissible highest--weights
\[\lambda' =\Big(\underbrace{k+x^1_1,k+x^2_2,\dots,k+x^q_q,\sum_{r=1}^q x_{q+1}^r, \sum_{r=2}^q x^r_{q+2},\dots,x^q_{2q}}_{2q},0,\dots,0\Big).\]
To summarise, from our previous computation
\begin{eqnarray*}
\lambda'_i = 2k - \sum_{r=q+1-i}^q x^q_{q+r}\qquad (i=1,\dots,q)\\
\lambda'_{q+j} = \sum_{r=j}^q x^q_{q+r}\qquad (j=1,\dots, q)
\end{eqnarray*}
Finally, renaming the variables to define $i_{\alpha}=x^q_{q+\alpha}$ with $\alpha=1,2,\dots, q$ the statement of the Lemma follows.
\end{proof}

It will be convenient, to achieve greater simplicity in later expressions, to rearrange the variables so to have increasing indices in  Eqn. (\ref{ClebschGordan}). Hence, define $j_\alpha=i_{q-\alpha},\; \alpha=0,1,\dots, q-1$ such that
\begin{eqnarray*}
& &\mathbf{F}(i_{q-1},i_{q-2},\dots, i_1,2k-2\sum_{\alpha=1}^q i_\alpha, i_1, i_2,\dots, i_q,0,\dots, 0)\\
&=& \mathbf{F}(j_1,j_2,\dots, j_{q-1},2k-2\sum_{\alpha=0}^{q-1} j_\alpha, j_{q-1},j_{q-2},j_0,0,\dots,0)\\
&=&  {\mathbf F}\left(\sum_{i=1}^{q-1}j_i \varpi_i + \left(2k-2\sum_{\alpha=0}^{q-1} j_\alpha \right)\varpi_q +\sum_{i=1}^qj_{q-i}\varpi_{q+i}\right).
\end{eqnarray*}
\medskip

\begin{lemma}\label{lowestF}
The lowest weight of 
\[ \mathbf{F}\left(\sum_{i=1}^{q-1}j_i \varpi_i + \left(2k-2\sum_{\alpha=0}^{q-1} j_\alpha \right)\varpi_q +\sum_{i=1}^qj_{q-i}\varpi_{q+i}\right)\] is
\[ -2k+ \frac{p+q}{pq}\sum_{i=1}^{q-1}(q-i)j_i +\left(1+\frac{q}{p}\right)j_0\]
\end{lemma}

\begin{proof}
If $\check w$ is the lowest  weight vector in the aforesaid $\mathrm{SU}(p+q)$ representation, then following (\ref{actionofY})
\begin{eqnarray*}
Y\cdot \check w & = & y^{-\frac1q \sum_{i=1}^q i k_i-\frac1p \sum_{j=1}^{p-1}(p-j)	k_{q+j}}\check w\\
&=& 
y^{-\frac1q\sum_{\alpha=1}^{q-1}\alpha j_\alpha - (2k-2 \sum_{\alpha=0}^{q-1}j_\alpha) -\frac1p \sum_{\alpha=1}^{q-1}(p-\alpha)j_{q-\alpha}-\frac{p-q}{p}j_0}\check w\\
&=& 	y^{-2k-\frac1q\sum_{\alpha=1}^{q-1}\alpha j_\alpha + 2\sum_{\alpha=0}^{q-1} j_\alpha-\frac1p\sum_{l=1}^{q-1}(p-(q-l)) j_l - \frac{p-q}{p}j_0}\check w\\
&=& y^{-2k - \sum_{i=1}^{q-1}\left(\frac{i}{q}-2+\frac{p-q+i}{p}\right)j_i+\left(1+\frac{q}{p}\right)j_0} \check w. 
\end{eqnarray*}
We can see that
\begin{eqnarray*}
& &\sum_{i=1}^{q-1}\left(\frac{i}{q}-2+\frac{p-q+i}{p}\right)j_i = \sum_{i=1}^{q-1}\left(\frac{i}{q}-1+\frac{-q+i}{p}\right)j_i\\
&=&
\sum_{i=1}^{q-1	}\left(\frac{pi-pq+q(-q+i)}{pq}\right)j_i
=
\sum_{i=1}^{q-1}\left(\frac{(p+q)i-q(p+q)}{pq}\right)j_i\\
&=&\frac{p+q}{pq}\sum_{i=1}^{q-1}(i-q)j_i.
\end{eqnarray*}
Hece, 
\[Y\cdot\check w = y^{-2k+\frac{p+q}{pq}\sum_{i=1}^{q-1}(q-i)j_i+\left(1+\frac{q}{p}\right)j_0}\check w.\]
\end{proof}

In order to describe explicitly the moduli up to gauge equivalence of maps of holomorphic isometric embeddings
$\mathrm{Gr}_p(\mathbf{C}^{p+q})\to \mathcal Q_N=\mathrm{Gr}_N(W)$  
of degree $k,$ with $p\geq q,$ the generalisation of do Carmo--Wallach theory,  Theorems \ref{GenDW} and  \ref{quadmodcpx},    requires to determine  $\G\mathrm{S}(\mathfrak{m}V_0,V_0)\cap \mathrm  H(W)^\perp $.

 The universal quotient bundle  $Q\to \mathrm{Gr}_p(\mathbf C^{p+q})$ of the $p$--plane Grassmannian  is of rank $q,$ and therefore is related to the holomorphic line bundle $\mathcal O(1)\to \mathrm{Gr}_p(\mathbf C^{p+q})$  
by $\Lambda^q Q \cong \mathcal O(1).$ At the reference point $o\in\mathrm{Gr}_p(\mathbf C^{p+q}),$ the center $\mathrm U(1)$ of the isotropy subgroup $\mathrm S(\mathrm U(p)\times\mathrm U(q)),$ generated by $Y,$ acts on the fibre $\mathcal O(1)_o$ with  weight $-1.$ We will denote this standard
fibre by $\mathbf C_{-1}.$ Analogously, $\mathbf{C}_k$ will stand for the  irreducible representation of $\mathrm U(1)$ with highest weight $k.$  
With this notation the homogeneous bundle $\text{SU}(p+q)\times_{\mathrm  S(\mathrm U(p)\times \mathrm U(q))} \mathbf{C}_{-k}$ is 
the complex line bundle $\mathcal O(k) \to  \mathrm{Gr}_p(\mathbf{C}^{p+q}),\;p\geq q,$ of degree $k$. 
Notice that, by  Theorem \ref{GenDW}, $W$ is identified with $H^0(\mathrm{Gr}_p(\mathbf{C}^{p+q}),\mathcal O(k)).$

\begin{lemma}\label{mV0}
\[\mathfrak m V_0 = \mathbf{C}^p\otimes \mathbf{C}^{q*}\otimes \mathbf{C}_{-k+\frac1p+\frac1q}.\]
\end{lemma}

\begin{proof}
If $\mathbf{C}^{p}=\mathbf{V}(1,0,\dots,0)$ denotes the standard representation of $\text{U}(p+q),$  
the tautological vector bundle $S\to \mathrm{Gr}_{p}(\mathbf{C}^{p+q})$ can be identified with the homogeneous bundle ${\rm SU}(p+q)\times_{\mathrm  S(\mathrm U(p)\times \mathrm U(q))}\mathbf{C}^p.$  The universal quotient bundle $Q\to \mathrm{Gr}_p\mathbf{C}^{p+q}$ is defined in the standard way, by the exact
sequence $0\to S\to \underline{\mathbf{C}^{p+q}}\to Q\to 0.$ Altogether, they define the holomorphic tangent bundle, $T\to \mathrm{Gr}_p(\mathbf{C}^{p+q}),\;$ $T=S^*\otimes Q,$ which is  a homogeneous bundle with standard fibre
\[ \left({\mathbf{C}^p} \otimes \mathbf{C}_{\frac{1}{p}}\right) ^{\ast} \otimes 
\left(\mathbf{C}^q \otimes 
\mathbf{C}_{-\frac{1}{q}}\right)
={\mathbf{C}^p}^{\ast} \otimes 
\mathbf{C}^q \otimes \mathbf{C}_{-\frac{1}{p}-\frac{1}{q}}.
\]
Analogously, the standard fibre of the cotangent bundle $T^*\to \mathrm{Gr}_p(\mathbf{C}^{p+q})$ is given by
\[T^{\ast}
={\mathbf{C}^p}\otimes 
{\mathbf{C}^q}^{\ast}\otimes \mathbf{C}_{\frac{1}{p}+\frac{1}{q}}.
\]

Now,  $T ^c\cong \mathfrak m^c,$ and $V_0$ is spanned by
the lowest--weight vector 	of $H^0(\mathrm{Gr}_p(\mathbf{C}^{p+q}),\mathcal O(k)),$ then
\[T^c\otimes V_0=
{\mathbf{C}^p}\otimes 
{\mathbf{C}^q}^{\ast}\otimes \mathbf{C}_{\frac{1}{p}+\frac{1}{q}}
\otimes \mathbf{C}_{-k}
={\mathbf{C}^p}\otimes 
{\mathbf{C}^q}^{\ast}\otimes \mathbf{C}_{-k+\frac{1}{p}+\frac{1}{q}}.
\]
\end{proof}

\begin{lemma}\label{intersection}
\[{\rm GS}(\mathfrak m V_0,V_0)\cap  \mathrm  H(W)^\perp=\mathbf{F}(2k\varpi_q).\]
\end{lemma}

\begin{proof} 
As a consequence of the previous Lemma \ref{mV0}, the lowest--weight of $\mathrm{GS}(\mathfrak{m} V_0,V_0)$  is
\[
-k-k+\frac{1}{p}+\frac{1}{q}=-2k+\frac{1}{p}+\frac{1}{q}. 
\]
On the other hand, by Lemma \ref{lowestF}, the lowest--weight of
\[ F\left(\sum_{i=1}^{q-1}j_i \varpi_i + \left(2k-2\sum_{\alpha=0}^{q-1} j_\alpha \right)\varpi_q +\sum_{i=1}^qj_{q-i}\varpi_{q+i}\right)\] is lower than the lowest--weight of $\mathrm{GS}(\mathfrak m V_0,V_0)$ in $\mathrm S^2(\mathbf{F}(k \varpi_q))$ if and only if
\[ -2k+ \frac{p+q}{pq}\sum_{i=1}^{q-1}(q-i)j_i +\left(1+\frac{q}{p}\right)j_0 \leq -2k +\frac1p +\frac1q, \]
or, equivalently, if and only if
\[ \frac{p+q}{pq}\left(-1+\sum_{i=1}^{q-1}(q-i)j_i \right)+\left(1+\frac{q}{p}\right)j_0\leq0,\]
if and only if
\[-1+\sum_{i=1}^{q-1}(q-i)j_i  \leq 0,\quad j_0=0,\]
that is
\[(q-1)j_1+(q-2)j_2+\dots+2 j_{q-2}+j_{q-1}\leq 1,\quad j_0=0.\]
It follows that $j_{q-1}=\{0,1\}$ and all the other terms vanish. 
Therefore,
\[\mathrm{GS}(\mathfrak m V_0,V_0)\subseteq  \underbrace{\mathbf{F}(2k\varpi_q)}_{j_{q-1}=0}\oplus \underbrace{ \mathbf{F}((2k-2)\varpi_q+\varpi_{q+1})}_{j_{q-1}=1}.\]
However, since 
\begin{equation}\label{sym}\mathbf{F}(2k \varpi_q)\subset \mathrm S^2(\mathbf{F}(k\varpi_q))\end{equation}
while 
\begin{equation}\label{skew}\mathbf{F}((2k-2)\varpi_q+\varpi_{q+1})\subset\Lambda^2(\mathbf{F}(k\varpi_q))\end{equation}
we have that
\[\mathrm{GS}(\mathfrak m V_0,V_0)\cap\mathrm{H}(W)^\perp=\mathbf{F}(2k \varpi_q).\]

The inclusions \ref{sym} and \ref{skew} are determined as follows: Let $\hat w$ denote as usual the highest--weight vector of $\mathbf{F}(k\varpi_q),$ and denote by $\hat w_-$ the vector in $\mathbf{F}(k\varpi_q)$
with weight just below that of $\hat w.$ It is clear that $\hat w\otimes \hat w$ is contained in $\mathrm{S}^2 \mathbf{F}(k\varpi_q)$ and that it is the highest--weight vector in the irreducible component   $\mathbf{F}(2k\varpi_q).$ Therefore, the inclusion \ref{sym} follows. Analogously, $\hat w \wedge \hat w_-$ belongs
to $\wedge^2(\mathbf{F}(k\varpi_q)),$ and it is the highest--weight vector in the
irreducible component $\mathbf{F}((2k-2)\varpi_q+\varpi_{q+1}),$ implying the inclusion \ref{skew}, and completing the proof.
\end{proof}

\begin{cor}\label{intersection 2} 
The orthogonal complement to 
$\G\mathrm{S}(\mathfrak{m}V_0,V_0) \oplus \mathbf{R}\,Id$
in $\mathrm{Sym}(W)$ is 
\[
V_k=
\mathrm S^2(\mathbf{F}(k\varpi_q))-\mathbf{F}(2k\varpi_q)
\]
\end{cor}

\begin{rem}
Although hook--length formulae for the dimension $\mathbf d_{p+q}(\lambda)$ of each $\mathrm{SU}(p+q)$--representation
$\mathbf{V}(\lambda)$  contributing for the moduli exist, the explicit expression is too cumbersome to be useful in itself. The dimension of the corresponding moduli is more easily computed simply
as $\dim V_k = \dim \mathrm S^2 \mathbf{F}(k \varpi_q) -  \dim \mathbf{F}(2k \varpi_q);$ explicitly
\begin{equation}\label{dimension}\dim V_k = \frac12 \mathbf d_{p+q}(k\varpi_q) (\mathbf{d}_{p+q}(k\varpi_q)+1) - \mathbf{d}_{p+q}(2k\varpi_q)\end{equation}
where the needed hook--rule formula is given by
\[\mathbf{d}_{a}(b\varpi_c)=\prod_{i=1}^{c}\prod_{j=1}^{b}\frac{a-i+j}{1+(c-i)+(b-j)}.\]
\end{rem}

As a consequence of the previous Lemmas and Corollary, Theorem \ref{GenDW} allows to give a clear picture of the desired moduli space, information which is summarised in the next Theorem, and which generalises to arbitrary Grassmannian manifolds the main result in \cite{MaNa21}:

\medskip

\begin{thm}\label{gmod1} 
 If $f:\mathrm{Gr}_p(\mathbf{C}^{p+q}) \to \mathrm{Gr}_n (\R^{n+2})$ is a full holomorphic isometric embedding of degree $k,$ then 
$n+2 \leq \dim V_k$.

Let $\mathcal M_k$ be the moduli space of full holomorphic isometric embeddings of degree $k$
  of $\mathrm{Gr}_p(\mathbf{C}^{p+q})$ into $\mathrm{Gr}_{N}(\mathbf{R}^{N+2})$
 by the gauge equivalence of maps, 
where $N+2=\dim V_k. $
Then, $\mathcal M_k$ can be regarded as an open bounded convex body 
in $V_k$.

Let $\overline{\mathcal M_k}$ be the closure of the moduli $\mathcal M_k$ by  the
topology induced from 
the inner product. 
Every boundary point of $\overline{\mathcal M_k}$ 
distinguishes a subspace $\mathbf{R}^{h+2}$ of $\mathbf{R}^{N+2}$ and 
describes a full holomorphic isometric embedding into $\mathrm{Gr}_h(\mathbf{R}^{h+2})$ 
which can be regarded as totally geodesic submanifold of 
$\mathrm{Gr}_{N}(\mathbf{R}^{N+2})$. 
The inner product on $\mathbf{R}^{N+2}$ determines 
the orthogonal decomposition of $\mathbf{R}^{N+2}:$ 
$\mathbf{R}^{N+2}=\mathbf{R}^{h+2}\oplus (\mathbf{R}^{{h+2}})^{\bot}$. 
Then the totally geodesic submanifold $\mathrm{Gr}_h(\mathbf{R}^{h+2})$ 
can be obtained as the common zero set of sections of 
$Q\to \mathrm{Gr}_{N}(\mathbf{R}^{N+2})$,
which belongs to $(\mathbf{R}^{{h+2}})^{\bot}$.  
\end{thm}

The formal relation of this theorem to Theorem \ref{GenDW}, and
consequently its proof, is the same as that of the Main Theorem (Theorem 3.6) in \cite{MaNa21}. Nevetheless, we reproduce it here with minor changes for the sake of
clarity and completeness.

\begin{proof}
The constraint $n\leq N$ is a consequence of (a) in Theorem \ref{GenDW} and Bott--Borel--Weil theorem.

It follows from  (c) 
in Theorem \ref{GenDW}  that $\G\mathrm{S}(\mathfrak{m}V_0,V_0)^\perp$ is a parametrization
of the space of full holomorphic isometric embeddings 
$f: \mathrm{Gr}_p\mathbf{C}^{p+q}
\to \mathrm{Gr}_{N}(\mathbf{R}^{N+2})$ of degree $k.$ 
Since the standard map into $\mathrm{Gr}_{N}(\mathbf{R}^{N+2})$ 
is the composite of the Kodaira embedding
$\mathbf \mathrm{Gr}_m(\mathbf{C}^{m+2}) \to \mathbf{C}P^{\frac{N}{2}}$ and 
the totally geodesic embedding 
$\mathbf{C}P^{\frac{N}{2}} \hookrightarrow \mathrm{Gr}_{N}(\mathbf{R}^{N+2})$,  
we can apply Theorem \ref{quadmodcpx} and Corollary \ref{intersection 2} 
to conclude that $\mathcal M_k$ is 
a bounded connected \emph{open} convex body in  $V_k$ with the
topology induced by the $L^2$ scalar product.

Under the natural compactification in the $L^2$-topology, 
the boundary points correspond to endomorphisms $T$ which are not positive definite,
but positive semi-definite. 
It follows from Theorem \ref{GenDW} that each of these endomorphisms defines 
a full holomorphic isometric embedding 
$\mathrm{Gr}_p(\mathbf{C}^{p+q}) \to \mathrm{Gr}_{h}(\mathbf{R}^{h+2})$,  
of degree $k$ with $h=2k-{\rm dim}\;\mathrm{Ker}\;T,$ whose target embeds in $\mathrm{Gr}_{N}(\R^{N+2})$ 
as a totally geodesic submanifold. 
The image of the embedding $\mathrm{Gr}_h(\R^{h+2}) \hookrightarrow \mathrm{Gr}_{N}(\R^{N+2})$ 
is determined by the common zero set of
sections in  $\mathrm{Ker}\;T$. 
(See also the Remark after Proposition 5.14 in \cite{Na13} for the geometric meaning of 
the compactification of the moduli space.)  
\end{proof}

\begin{rem} 
It follows 
from Corollary 5.18 in \cite{Na15} that 
the first condition in (\ref{HDW 2}) is automatically satisfied.  
Alternatively, using the same techniques as in Lemma  \ref{intersection} it can be shown that
\[{\rm GS}(V_0,V_0)\cap {\rm H}(W)^\perp = \mathbf{V}(2k,2k,0,\dots,0).\]
\end{rem}

The centraliser $S^1\cong \mathrm U(1)$ of the holonomy subgroup $\mathrm K=\mathrm  S(\mathrm U(m)\times \mathrm U(2))$ of the structure group of the line bundle acts on $\mathcal M_k.$  
For the general theory, see \cite{Na15}.  Therefore, the same
argument and proof as the one of Theorem 8.1 in \cite{MNT} applies leading to 

\begin{cor}\label{imod}
The  moduli space $\mathbf M_k$ of image equivalence classes of holomorphic isometric embeddings  $\mathrm{Gr}_m(\mathbf{C}^{m+2})\to \mathrm{Gr}_{N}(\mathbf{R}^{N+2}), N+2=\dim V_k,$ 
of degree $k,$  
is
\[\mathbf M_k=\mathcal M_k\slash S^1.\] 
\end{cor}

\begin{rem}
There is a natural, induced complex structure defined on $\mathcal M_k$ from its embbedding in $ V_k.$ It is also equipped with a compatible metric induced from the inner product, so $\mathcal M_k$ is a K\"ahler manifold.
The aforesaid $S^1$--action preserves the K\"ahler structure on $\mathcal M_k.$ The moment map $\mu:\mathcal M_k\to \mathbf{R}\;:\; |Id-T^2|^2$ induces the K\"ahler quotient, and 
$\mathbf M_k$ has a foliation   
whose general leaves are the complex projective spaces. 
\end{rem}

\section{A worked--out example: embedding of $\mathrm Gr_{q}(\mathbf C^{2q})$}

This example generalises the case of the holomorphic isometric 
embedding of the complexified, compactified Minkowski Space $\mathrm Gr_2(\mathbf C^4)$ into quadrics,  discussed in \cite{MaNa21}. For that end, assume $p=q$ in the general discussion above. Then, we have that
\begin{eqnarray*}
\otimes^2 \mathbf{F}(k\varpi_q) 
& = & \bigoplus_{i_1,i_2,\dots, i_q=0}^{k\geq \sum_{\alpha=1}^q i_\alpha}\mathbf{V}(2k-i_q, 2k-(i_{q-1}+i_q),\dots,2k-\sum_{\alpha=1}^q i_\alpha,\\
& & \qquad\qquad 
\sum_{\alpha=1}^q i_\alpha, \sum_{\alpha=2}^q i_\alpha,\dots, i_{q-1}+i_q, i_q)\\
&=& 
\bigoplus_{i_\alpha=0}^{k\geq \sum_{\alpha=1}^q i_\alpha} 
\mathbf{F}(
\underbrace{
i_{q-1},i_{q-2},\dots,i_2,i_1}_{q-1},2k-2\sum_{\alpha=1}^q i_\alpha,
\underbrace{i_1,i_2,\dots, i_{q-1}}_{q-1})\\
&=&
\bigoplus_{j_\alpha=0}^{k\geq \sum_{\alpha=0}^{q-1} j_\alpha}  F\left(\sum_{i=1}^{q-1} j_i \varpi_i + \left(2k-2\sum_{\alpha=0}^{q-1} j_\alpha\right)\varpi_q+ \sum_{i=1}^{q-1}j_{q-i}\varpi_{q+i}\right).
\end{eqnarray*}
Comparing this last expression with the general one for the case $p>q,$ the difference is the vanishing of the term $j_0 \varpi_{2q}.$ Hence,
\[Y\cdot \check w =  y^{-2k +\frac{p+q}{pq}\sum_{i=1}^{q-1}(q-i)j_i+2j_0}\check w.\]
The term $2j_0$ is the same as $(1+\frac{q}{p})j_0$ when $p=q.$ Hence, if $p=q$ the content of Lemma \ref{lowestF} would be restated as

\begin{lemma}
The lowest weight of 
\[F\left(\sum_{i=1}^{q-1} j_i \varpi_i + \left(2k-2\sum_{\alpha=0}^{q-1} j_\alpha\right)\varpi_q+ \sum_{i=1}^{q-1}j_{q-i}\varpi_{q+i}\right)\]
is
\[-2k + \frac2q \sum_{i=1}^{q-1}(q-i)j_i + 2j_0.\]
\end{lemma}
Now, from Lemma \ref{mV0}, 
\[\mathfrak m  V_0 =\mathbf C^q\otimes (\mathbf C^q)^* \otimes \mathbf C_{-k+\frac2q}.\]
Therefore, $\mathrm{GS}(\mathfrak m V_0, V_0)$ has the weight
\[-k-k+\frac2q = -2k+\frac2q.\]

The lowest weight of 
\[F\left(\sum_{i=1}^{q-1} j_i \varpi_i + \left(2k-2\sum_{\alpha=0}^{q-1} j_\alpha\right)\varpi_q+ \sum_{i=1}^{q-1}j_{q-i}\varpi_{q+i}\right)\] is lower than the one of $\mathrm{GS}(\mathfrak m V_0,V_0)$ in $\mathrm S^2(\mathbf{F}(k \varpi_q))$ if and only if 
\[-2k+\frac2q\sum_{i=1}^{q-1}(q-i)j_i+2j_0 \leq -2k + \frac2q,\]
equivalently
\[\frac2q\left(-1 + \sum_{i=1}^{q-1}(q-i)j_i\right) + 2j_0\leq 0\]
if and only if
\[\sum_{i=1}^{q-1}(q-i)j_i-1\leq 0,\qquad j_0=0.\]
Thus,
\[(q-1)j_1+(q-2)j_2+\dots+ 2j_{q-2}+j_{q-1}\leq 1,\qquad j_0=0.\]
It follows that $j_{q-1}\in \{0,1\}$ and the others terms must vanish. As a consequence,
\[\mathrm{GS}(\mathfrak m V_0,V_0)\subset 
\underbrace{\mathbf{F}(2k \varpi_q)}_{j_{q-1}=0}\oplus 
\underbrace{\mathbf{F}((2k-2)\varpi_q+\varpi_{q+1})}_{j_{q-1}=1}\]
and then again, since $\mathbf{F}(2k\varpi_q)\in \mathrm S^2(\mathbf{F}(k\varpi_q))$ but $\mathbf{F}((2k-2)\varpi_q+\varpi_{q+1})\subset \Lambda^2(\mathbf{F}(k \varpi_q))$ we have that
\[\mathrm{GS}(\mathfrak m V_0,V_0)=\mathbf{F}(2k\varpi_q).\]

The real dimension of $V_k$ is therefore computed
as in the general case Eqn. (\ref{dimension}) where now $n=2q,$
that is, using
\[\mathbf d_{2q}(k\varpi_q)=\prod_{i=1}^q\prod_{j=1}^k \frac{2q-i+j}{1+(q-i)+(k-j)},
\quad
\mathbf d_{2q}(2k\varpi_q)=\prod_{i=1}^q\prod_{j=1}^{2k} \frac{2q-i+j}{1+(q-i)+(2k-j)}.
\]
Moreover, if $q$ is even, say $q=2m,$ then $\Lambda^q\mathbf C^{2q}=\Lambda^{2m}\mathbf C^{4m}$ has an invariant real structure, which will be denoted by $\sigma,$ determining a real subspace $U=(\Lambda^{2m}\mathbf{C}^{4m})^{\mathbf R}\subset \Lambda^{2m}\mathbf{C}^{4m}.$ Therefore, if $k=1$ we obtain a one--parameter family of holomorphic isometric embeddings generalising  Corollary 3.7 of \cite{MaNa21}
\begin{cor}
There exists a one--parameter family $\{f_t\}$ with $t\in [0,1]$ of
$\mathrm{SU}(4m)$--equivariant non--congruent holomorphic isometric embeddings of degree 1
of $\mathrm{Gr}_{2m}(\mathbf C^{4m})$ into complex quadrics. The mapping $f_0$ corresponds to the Kodaira
embedding, while $f_1$ is the 
real standard map defined by the pair $(\mathcal{O}(1),U).$
\end{cor}

\end{document}